\documentclass[a4paper,12pt]{article}
\usepackage[utf8]{inputenc}
\usepackage{amssymb,amsthm,amsmath}

\usepackage{titlesec}

\titleformat*{\section}{\bfseries}

\usepackage[margin = 2.5cm]{geometry}

\newtheorem{thm}{Theorem}
\newtheorem{lmm}[thm]{Lemma}

\newcommand{\RR}{\mathbb{R}}
\newcommand{\vl}{^\mathsf{v}}
\newcommand{\tl}{^\mathsf{c}}
\newcommand{\slt}{\mathring{T}}
\newcommand{\vm}{\mathfrak{X}}

\newcommand{\ir}[1]{\mathcal{#1}}
\newcommand{\hll}[1]{\widetilde{#1}}
\newcommand{\prcf}[3]{
\frac{\partial #1}{\partial #2^{#3}}
}

\let\OLDthebibliography\thebibliography
\renewcommand\thebibliography[1]{
  \OLDthebibliography{#1}
  \small
  \setlength{\parskip}{0pt}
  \setlength{\itemsep}{0pt plus 0.3ex}
}

\begin{document}
\title{\vspace{-6ex}A generalization and short proof of a theorem of Hano on affine vector fields\vspace{-1ex}}
\author{D.~Cs.~Kert\'esz \and R.~L.~Lovas\footnote{Both authors were supported by the Hungarian Scientific Research Fund (OTKA) Grant K-111651}\vspace{-2ex}}
\maketitle

\vspace{-6ex}

\begin{abstract}
We prove that a bounded affine vector field on a complete Finsler manifold is a Killing vector field. This generalizes the analogous result of Hano for Riemannian manifolds \cite{MR0076394}. Even though our result is more general, the proof is significantly simpler.
\end{abstract}
{\it Keywords:} affine vector field, Killing field, Finsler manifold.\\
{\it 2010 Mathematics Subject Classification:} 53B40, 53C60.

\section{Introduction}
Yano showed that affine vector fields on a compact orientable Riemannian manifold are Killing fields \cite{Yanokilling}; the proof was based on integral formulas. Hano found a generalization: bounded affine vector fields on a complete Riemannian manifold are Killing fields. The proof relied on the de Rham decomposition, and special properties of irreducible Riemannian manifolds. A similar proof can be found in \cite{MR1393940}. We show that Hano's result is true for the much more general Finsler manifolds, using only the Euler--Lagrange equation.

\section{Definitions and prerequisites}

Throughout, $M$ is a second countable and smooth Hausdorff manifold; $\tau\colon TM\to M$ is the tangent bundle, $\mathring\tau\colon \slt M\to M$ is the slit tangent bundle of $M$. If $\varphi\colon M\to N$ is a smooth mapping between manifolds, $\varphi_\ast\colon TM\to TN$ stands for its derivative. 

We are going to work on the tangent manifold, where we use two kinds of lifts of vector fields on the base manifold. The \emph{vertical lift} $X\vl$ of a vector field $X\in\vm(M)$ is the velocity field of the global flow
\[
 (t,v)\in\RR\times TM\longmapsto v+tX(\tau(v)) \in TM
\]
on $TM$. If $\varphi^X\colon \ir D_X\subset\RR\times M\to M$ is the maximal local flow of $X\in\vm(M)$ and
 \[
   \hll{\ir D}_X:=\{(t,v)\in\RR\times TM\mid (t,\tau(v))\in\ir D_X\},
 \]
then
\[
 (t,v)\in \hll{\ir D}_X \longmapsto(\varphi^X_t)_*(v) \in TM
\]
is a local flow on $TM$, whose velocity field is called the \emph{complete lift} of $X$, denoted by $X\tl$. The \emph{Liouville vector field} $C$ on $TM$ is the velocity field of the flow of positive dilations:
\[
 (t,v)\in\RR\times TM\longmapsto e^tv\in TM.
\]
It is clear that a smooth function $f$ on $\slt M$ is $k^+$-homogeneous ($k\in\mathbb Z$) if and only if $Cf = kf$.

A continuous function $F$ on $TM$ is a \emph{Finsler function} for $M$ if it is smooth on $\slt M$, $1^+$-homogeneous, $F\upharpoonright \slt M >0$, and for any $p\in M$ and $u\in \slt_pM$, the symmetric bilinear form $(E\upharpoonright T_pM)''(u)$ is non-degenerate (hence positive definite), where $E:=\frac 12F^2$. A \emph{Finsler manifold} is a manifold together with a Finsler function. 

If $(M,F)$ is a Finsler manifold, then there exists a unique second-order vector field $S\in\vm(\slt M)$ such that a curve $\gamma$ in $M$ is a geodesic of $(M,F)$ if and only if $S\circ\dot\gamma=\ddot\gamma$. Another characterization of $S$ is that
\begin{equation}\label{eq:EL}
S(X\vl E) - X\tl E = 0\qquad \mbox{for all } X\in\vm(M).
\end{equation}
This form of the Euler--Lagrange equation is due to Crampin (see, e.g., \cite[p.~348]{MR892315} or \cite[p.~16]{SZLK11}). It can be derived directly from the elementary form $\left(\prcf Eyi\circ\dot\gamma\right)'-\prcf Exi\circ\dot\gamma=0$ using the local formulae for $X\vl$ and $X\tl$. 

A Finsler manifold is said to be \emph{forward complete} if the domains of its maximal geodesics are not bounded from above, and \emph{complete} if the domain of its maximal geodesics is $\RR$. For many equivalent characterizations of completeness, see \cite[\S 6.6]{BCS}.

A vector field $X$ on a Finsler manifold $(M,F)$ is \emph{affine} if its flow preserves geodesics, and it is a \emph{Killing field} if its flow preserves the Finsler function, i.e., $F\circ(\varphi^X_t)_*=F$  for all possible $t\in\RR$. Both properties can be expressed in terms of the complete lift of $X$: $X$ is affine if and only if $[X\tl,S]=0$, and $X$ is a Killing field if and only if $X\tl E =0$.

\section{Proof of the result}
The key of our argument is the following simple observation. It is in fact a disguised special case of Exercise 5.4.3 from \cite{BCS}.
\begin{lmm}\label{lmm:XveXce}
 If $X$ is an affine vector field on a Finsler manifold $(M,F)$ and $\gamma$ is a geodesic, then for all $t$ and $t_0$ in the domain of $\gamma$ we have 
 \[
 X\vl E(\dot\gamma(t)) = X\vl E(\dot\gamma(t_0)) + (t-t_0) X\tl E(\dot\gamma(t_0)).
 \]
\end{lmm}

\begin{proof}
 Since $X$ is affine, we have $[X\tl,S]=0$. Geodesics have constant speed, hence $SE=0$. From these we get
 \[
  0=[X\tl,S]E = X\tl(SE) -S(X\tl E)=-S(X\tl E).
 \]
 Since $\gamma$ is a geodesic, $S\circ\dot\gamma=\ddot\gamma$, and we have
\begin{align*}
 (X\vl E\circ\dot\gamma)' &= S(X\vl E)\circ\dot\gamma\overset{\eqref{eq:EL}}=X\tl E\circ\dot\gamma,\\
 (X\vl E\circ\dot\gamma)'' &= (X\tl E\circ\dot\gamma)' = S(X\tl E)\circ\dot\gamma=0.
\end{align*}
Therefore $X\vl E\circ\dot\gamma$ is an affine function, and our claim follows.
\end{proof}

\begin{thm}
Let $(M,F)$ be a Finsler manifold, $X$ an affine vector field, and suppose that one of the following conditions holds:
\begin{itemize}
 \item[(1)] $F \circ X$ is bounded, and $(M,F)$ is complete;
 \item[(2)] $F \circ X$ and $F \circ (-X)$ are bounded, and $(M,F)$ is forward complete.
\end{itemize}
Then $X$ is a Killing vector field.
\end{thm}

\begin{proof}
First we prove that $X\vl E$ is bounded from above on $U(TM):=F^{-1}(\{1\})$ if (1) holds, and it is bounded from above and from below if (2) holds. For any $v\in U(TM)$, setting $p:=\tau(v)$, we have
\[
X\vl E(v) =   F(v) X\vl F(v) = X\vl F(v) = (F\upharpoonright T_pM)'(v)(X(p))\leq F(X(p)),
\]
where in the last step we used the fundamental inequality (see \cite[p.~7]{BCS} or \cite[Proposition~9.1.37]{CSF}). In a similar way, we obtain
\[
X\vl E(v) = (F \upharpoonright T_p M)'(v)(X(p)) = -(F \upharpoonright T_pM)'(v)(-X(p))\geq -F (-X(p)). 
\]
Over $U(TM)$ these two inequalities give $-F \circ (-X) \circ \tau\leq X\vl E\leq F \circ X \circ \tau$, from which our first assertion follows.

Now we show that $X\tl E=0$, and hence $X$ is a Killing field. It suffices to prove it on $U(TM)$, because $X\tl E$ is $2^+$-homogeneous. Indeed, the flows of $X\tl$ and $C$ clearly commute, hence $[X\tl,C]=0$, and we have $C(X\tl E) = [C,X\tl]E + X\tl(C E) = 2X\tl E$. 

So fix $v\in U(TM)$ and let $\gamma$ be the maximal geodesic with $\dot\gamma(0)=v$. Then Lemma~\ref{lmm:XveXce} gives
 \[
  X\vl E(\dot\gamma(t)) = X\vl E(\dot\gamma(0))+ t X\tl E(\dot\gamma(0)) = X\vl E(v) + t X\tl E(v)
 \]
for any real number $t$ in case (1) and for any positive real number $t$ in case (2). Geodesics have constant speed, hence $\dot\gamma$ remains inside $U(TM)$, and the left-hand side of the above formula has to be bounded from above in case (1), and it has to be bounded from above and below in case (2), which is possible only if $X\tl E(v) = 0$.
\end{proof}
As an immediate corollary we have
\begin{thm}[Hano]
 A bounded affine vector field on a complete Riemannian manifold is a Killing field.
\end{thm}

Since compact Finsler manifolds are complete, we also have
\begin{thm}
 An affine vector field on a compact Finsler manifold is a Killing vector field.
\end{thm}

\section*{Acknowledgements}
The authors are grateful to J\'ozsef Szilasi and Bernadett Aradi for their suggestions that improved the paper.

.


\end{document}